\documentclass[11pt,a4paper,twoside]{article}
\usepackage{bm, amsmath, amssymb, amsthm} 

\def\RR{\mathbb{R}}

\newcommand\tr{\operatorname{trace}}
\newcommand\Div{\operatorname{div}}
\def\Ric{\operatorname{Ric}}
\def\vol{\operatorname{vol}}

\def\eq{\hspace*{-1.5mm}&=&\hspace*{-1.5mm}}

\newtheorem{corollary}{Corollary}

\newtheorem{remark}{Remark}
\newtheorem{lemma}{Lemma}

\newtheorem{theorem}{Theorem}

\author{
        Dhriti Sundar Patra\footnote{Department of Mathematics, Birla Institute of Technology Mesra, Ranchi: 835 215, India
       \newline e-mail: {\tt dhritimath@gmail.com} \ and \ {\tt dpatra.teqip@bitmesra.ac.in}}
        \ and \
       Vladimir Rovenski\footnote{Department of Mathematics, University of Haifa, Mount Carmel, 31905 Haifa,  Israel
       \newline e-mail: {\tt vrovenski@univ.haifa.ac.il}
       }
}

\title{On
non-gradient $(m,\rho)$-quasi-Einstein contact metric manifolds
}

\begin{document}

\date{}

\maketitle

\begin{abstract}
Many authors have studied Ricci solitons and their analogs within the framework of (almost) contact geometry.
In this article, we thoroughly study the $(m,\rho)$-quasi-Einstein structure on a contact metric manifold.
First, we prove that if a $K$-contact or Sasakian manifold $M^{2n+1}$ admits a closed $(m,\rho)$-quasi-Einstein structure, then it is an
Einstein manifold of constant scalar curvature $2n(2n+1)$,
and for the particular case -- a non-Sasakian $(k,\mu)$-contact structure --
it is locally isometric to the product of a Euclidean space $\RR^{n+1}$ and a sphere $S^n$ of constant curvature $4$. Next, we prove that if a compact contact or $H$-contact metric manifold admits an $(m,\rho)$-quasi-Einstein structure, whose potential vector field $V$ is collinear to the Reeb vector field, then it is a $K$-contact $\eta$-Einstein manifold.


\end{abstract}

\date{}

\section{Introduction}

Since the discovering of Ricci solitons as a fruitful generalization of Einstein manifolds,
many authors have studied Ricci solitons and their analogs within the framework of (almost) contact geometry.

A Riemannian manifold $(M^n,g)$ of dimension $n\geq3$ is said to be a \textit{generalized $m$-quasi-Einstein manifold} if there exist a vector field $V\in\mathfrak{X}_M$ and a smooth function $\beta\in C^{\infty}(M)$ such that
\begin{eqnarray}\label{1.1A}
 \frac{1}{2}\,{\cal L}_V\,g +\Ric -\frac{1}{m}\,V^{\flat}\otimes V^{\flat} = \beta\,g,
\end{eqnarray}
where $0<m\le\infty$ and $\Ric$ is the Ricci tensor of $g$; ${\cal L}_V$ denotes the Lie derivative operator
along $V$, and $V^{\flat}$ is the $1$-form metrically dual to
$V$. It was introduced and studied by Barros at. el. \cite{barros2014characterizations,barros2017triviality}, and they constructed a family of nontrivial generalized $m$-quasi-Einstein metrics on the unit sphere $S^n(1)$ that are rigid in the class of metrics with constant scalar curvature. Such a manifold is called \textit{trivial} if the vector field $V$ is identically zero on $M$, and the triviality condition implies that $M$ is an Einstein manifold: $\Ric=\beta\,g$, see \cite{besse2007einstein}.
When $\beta$ is constant then it becomes an $m$-\textit{quasi-Einstein manifold}, which was introduced and studied by Case et al. \cite{case2011rigidity}; they have natural geometric interpretations using warped product Einstein manifolds. Later on, Barros-Ribeiro \cite{barros2012integral,barros2014uniqueness} obtained some integral formulae and studied uniqueness of quasi-Einstein metrics on manifolds. Recently, Ghosh \cite{ghosh2019m} and Chen \cite{chen2020quasi} studied an $m$-quasi-Einstein structure on contact metric manifolds and on almost cosymplectic manifolds.

\smallskip

In this article, we consider the special case $\beta=\lambda+\rho\,r$ (which is also a function on $M$ depending on the scalar curvature $r$ of $(M^n,g)$) in \eqref{1.1A}.
Namely, a Riemannian manifold $(M^n,g)$ is said to be an $(m,\rho)$-\textit{quasi-Einstein manifold} if there exist a vector field $V\in\mathfrak{X}_M$ and real
$\lambda$ and $\rho$ such that
\begin{eqnarray}\label{1.2}
 \frac{1}{2}\,{\cal L}_V\,g +\Ric - \frac{1}{m}\,V^{\flat} \otimes V^{\flat} = (\lambda+\rho\,r)\,g
\end{eqnarray}
for some positive $m\in\RR$. Such $(M^n,g)$ is called a \textit{closed $(m,\rho)$-quasi-Einstein manifold} if the $1$-form $V^{\flat}$ is closed: $d\,V^{\flat}=0$.
In particular, if $m=\infty$, then \eqref{1.2} is exactly the $\rho$-\textit{Einstein soliton} \cite{catino2016gradient}, and for $\rho=0$, it reduces to a \textit{Ricci soliton} (cf. \cite{cao2009recent}). Thus, it is interesting to study $(m,\rho)$-quasi-Einstein manifolds as a generalization of Ricci solitons.
If the potential vector field $V$ is the gradient of a smooth function $f\in C^{\infty}(M)$, then the $(m,\rho)$-quasi-Einstein manifold is called a \textit{(gradient) $(m,\rho)$-quasi-Einstein manifold} (see \cite{ghosh2014m,huang2013classification,shin2017classification} for details).
In this case, \eqref{1.2} reduces to
                                                                                                                                                           \begin{eqnarray}\label{1.1}
 {\rm Hess}_f + \Ric - \frac{1}{m} df \otimes df = (\lambda+\rho\,r)\,g,
\end{eqnarray}
where ${\rm Hess}_f$ is the Hessian form of the smooth function $f$ on $M$.
The importance of such manifolds is explained by recent studies of the $m$-Bakry-Emery Ricci tensor $\Ric^m_f$ (see \cite{lott2003some,qian1997estimates}) defined by
\[
 \Ric^m_f={\rm Hess}_f + \Ric - \frac{1}{m} df \otimes df.
\]
Notice that if $m=\infty$ and $\rho=0$, then \eqref{1.1} reduces to a gradient Ricci soliton. The gradient Ricci soliton metrics play an
important role in the theory of Hamilton's Ricci flow since they correspond to self-similar solutions and often arise as singularity models, see a survey
by H.\,Cao in \cite{cao2009recent}. In~\cite{huang2013classification}, Huang-Wei considered \eqref{1.1} and proved some rigidity results by using the maximum principle and presented some classifications under Bach-flat condition.
Further on, Shin \cite{shin2017classification} studied four-dimensional (gradient) $(m,\rho)$-quasi-Einstein manifolds with harmonic Weyl curvature tensor when $m\not\in \{0,\pm 1,-2,\pm\infty,\}$ and $\rho\not\in\{\frac{1}{4},\frac{1}{6}\}$.

\smallskip

During the last few years, Ricci solitons and almost Ricci solitons have been studied by several authors
(see \cite{cho2010contact,ghosh2014sasakian,GSC,sharma2008certain}) within the framework of contact geometry.
In \cite{sharma2008certain}, Sharma initiated the study of gradient Ricci solitons within the framework of $K$-contact manifold and generalized the odd-dimensional Goldberg conjecture (proved by Boyer-Galicki \cite{boyer2001einstein}):
\textit{``Any compact K-contact Einstein manifold is Sasakian"}.
He proved that \textit{``any complete $K$-contact metric admitting a gradient Ricci soliton is compact, Einstein and Sasakian"}.
Later on, Ghosh et al. generalized this result for non-Sasakian $(k,\mu)$-spaces, complete $K$-contact gradient $(m,\rho)$-quasi-Einstein spaces and complete $K$-contact $m$-quasi-Einstein spaces with $m\ne1$ (see \cite{ghosh2019m,ghosh2014m,GSC}).
Further, Cho-Sharma \cite{cho2010contact} generalized this for homogeneous $H$-contact and compact contact manifolds, in particular, they proved that \textit{``if a $H$-contact metric manifold admits Ricci soliton with potential vector field $V$ collinear to the Reeb vector field, then it is Einstein and Sasakian"}. It is also true when $H$-contactness is replaced by compactness. So, it would be relevant to study non-gradient $(m,\rho)$-quasi-Einstein structure with a 1-form associated to the potential vector field in the framework of $K$-contact, Sasakian, $(k,\mu)$-contact, $H$-contact and compact contact manifolds in order to generalize the results of Boyer-Galicki \cite{boyer2001einstein}, Sharma \cite{sharma2008certain} and Ghosh \cite{ghosh2015quasi,ghosh2014m,ghosh2019m}.

The structure of this paper is the following.
In Section~2, we recall some basic definitions and fundamental formulas for contact metric manifolds.
In Section~3, we consider closed $(m,\rho)$-quasi-Einstein contact metric manifolds. In Section~4, we study $(m,\rho)$-quasi-Einstein contact metric manifolds, and then formulate our results and give their proofs.





\section{Notes on contact metric manifolds}

Here, we review some basic definitions and properties on (almost) contact metric manifolds,
see details in \cite{dileo2019generalizations,blair2010riemannian,blair1995contact}.
Let $T^{1}M$ be the unit tangent bundle of a compact orientable Riemannian manifold $(M,g)$
equipped with the Sasaki metric $g_{s}$.
Any unit vector field $U$ determines a smooth map between $(M,g)$ and $(T^{1}M, g_{s})$.
The {energy of the unit vector field} $U$ is defined by
\begin{eqnarray*}
 E: U \to \frac{1}{2}\int_{M}\parallel dU \parallel^{2} d\vol = \frac{n}{2}\,{\rm Vol}(M, g) + \frac{1}{2}\int_{M}\parallel \nabla U \parallel^{2} d\vol,
\end{eqnarray*}
where $dU$ denotes the differential of the map $U$, $\nabla$ is the Levi-Civita connection of $g$
and $d\vol$ denotes the volume element of $M$. Such a vector field $U$ called \textit{harmonic} if it is a critical point of the energy functional $E$ defined on the space $\mathfrak{X}^1$ of all unit vector fields on $(M,g)$.

An $(2n+1)$-dimensional Riemannian manifold $(M,g)$ endowed with a tensor field $\varphi$ of type $(1,1)$,
a unit vector field $\xi$ and a $1$-form $\eta$, is called a \textit{contact metric manifold} if these tensors satisfy \cite{blair2010riemannian}
\begin{equation}\label{2.1}
 \varphi^2 = -I + \eta\otimes\xi, \quad \eta = g(\cdot\,,\xi), \quad d\eta(X,Y)=g(X,\varphi\,Y)
\end{equation}
for any $X,Y\in\mathfrak{X}_M$, where $\mathfrak{X}_M$ is the Lie algebra of all vector fields on $M$.
The $1$-form $\eta$ is known as the \textit{contact form}, and $\xi$ is called the \textit{Reeb vector field}.
The structure $(\varphi,\xi,\eta,g)$ on $M$ is called a \textit{contact metric structure},
and the manifold is denoted by $M^{2n+1}(\varphi,\xi,\eta,g)$.
It follows from \eqref{2.1} that
\begin{equation*}
 \varphi(\xi) = 0,\quad
 \eta \circ \varphi = 0,\quad
 {\rm rank}(\varphi)=2n
\end{equation*}
and
\begin{equation}\label{Eq-Sas1}
 g(X, \varphi Y) =-g(\varphi X, Y) ,\quad
 g(\varphi X,\varphi Y)= g(X,Y) -\eta(X)\eta(Y)
\end{equation}
for any $X,Y\in\mathfrak{X}_M$.
The tangent bundle of $M$ splits as $TM=\mathcal{H}\otimes \big\langle\xi\big\rangle$, where $\mathcal{H}=\varphi(TM)=\ker\eta$ is a $2n$-dimensional distribution.
We consider two self-adjoint operators on $M$,
\begin{equation}\label{2.2}
 h=\frac{1}{2}\,{\cal L}_{\xi}\,\varphi,\quad
 l=R(\cdot,\xi)\,\xi,
\end{equation}
where
$R$ is the Riemann curvature tensor of $ g $. Operators \eqref{2.2} satisfy equalities, e.g., \cite[p.~84, p.~85]{blair2010riemannian}:
\begin{equation}\label{2.2b}
 \tr_g h=0, \quad \tr_g(h\,\varphi)=0, \quad h\xi=0, \quad l\,\xi = 0, \quad h\,\varphi=-\varphi\,h.
\end{equation}
The following formulas are valid on a contact metric manifold (see \cite[p. 84; p. 112]{blair2010riemannian}):
\begin{eqnarray}\label{2.3}
 &\nabla_{X}\,\xi = - \varphi\,X - \varphi\,h X,\\
\label{2.4b}
 &\Ric (\xi, \xi) = 2n - ||h||^2= \tr_g l
\end{eqnarray}
for any $X\in\mathfrak{X}_M$, where $Q$ the the Ricci operator associated with the Ricci $(0,2)$-tensor $\Ric$.
The~mean curvature vector of ${\cal H}$ is defined  for a local orthonormal frame $(E_i)$ of ${\cal H}$ by
\[
 H=-\Div\xi=-\sum\nolimits_{\,i} g(\nabla_{E_i}\,\xi,\,E_i).
\]

\begin{lemma}\label{L-01} For a contact metric manifold we have $H=0$.
\end{lemma}

\begin{proof} Using \eqref{2.3} and \eqref{2.2b}, for a local orthonormal frame $(E_i)$ of ${\cal H}$, we have
\begin{equation*}
 \sum\nolimits_{\,i} g(\nabla_{E_i}\,\xi,\,E_i) = -\tr\varphi -\tr\varphi\,h = \tr h\,\varphi =0 ,
\end{equation*}
thus the claim follows.
\end{proof}

A contact metric manifold $M$ is said to be a \textit{$K$-contact manifold} if $\xi$ is a Killing vector field (equivalently, $h=0$ or $\tr_gl=2n$).
 An $M^{2n+1}(\varphi,\xi,\eta,g)$ is called a \textit{Sasakian manifold} if
\[
 (\nabla_X \varphi)Y = g(X,Y)\xi -\eta(Y)X,\quad X,Y\in\mathfrak{X}_M.
\]
In this case, the following two important identities are satisfied by $\xi$,
see \cite
{blair2010riemannian}:
\begin{equation}\label{2.6}
 \nabla_{\xi}\,\xi = 0,\qquad
\nabla_{X}\,\xi = -\varphi\,X,\quad X\in\mathfrak{X}_M.
\end{equation}
A Sasakian manifold is $K$-contact but the converse is true only in dimension $3$, e.g., \cite[p. 87]{blair2010riemannian}.
Using \eqref{2.6}(b) and \eqref{Eq-Sas1}(a), we find
\begin{equation}\label{2.new1}
  g(X,\nabla_Y\,\xi)+g(Y,\nabla_X\,\xi) =0,\quad X,Y\in\mathfrak{X}_M,
\end{equation}
that is $\xi$ of a Sasakian manifold is a Killing vector field on $(M,g)$; thus, $g$ is a bundle-like metric with respect to flow of $\xi$.
Note that from the property \eqref{2.new1} it follows $H=0$.

 A contact metric manifold $ M^{2n+1}(\varphi, \xi, \eta, g)$ is said to be \textit{$(k, \mu) $-contact manifold} if its curvature tensor satisfies
\begin{equation}\label{2.10}
 R(X, Y)\,\xi = k\{\eta(Y)X - \eta(X)Y\} + \mu\{\eta(Y)\,hX - \eta(X)\,hY \}
\end{equation}
for all $X$, $Y\in\mathfrak{X}_M$ and for some real numbers $k$ and $\mu$. This class was introduced by Blair et al. in \cite{blair1995contact} and completely  classified by Boeckx in \cite {boeckx2000full}. This class arises by applying $\mathcal{D}$-homothetic deformations (e.g., Tanno \cite{tanno1968topology}; Blair \cite[p. 125]{blair1995contact}):
\[
 \bar{\eta} = a\eta, \quad \bar{\xi}= \frac{1}{a}\,\xi,\quad \bar{\varphi} = \varphi,\quad \bar{g} = ag + a(a-1)\eta\otimes\eta
\]
for a positive real constant $a$, to a contact metric manifold satisfying $ R(X, Y)\,\xi = 0$.
Note that the class of $(k,\mu)$-contact structure contains Sasakian manifolds (for $k = 1$) and the tangent sphere bundle (for $k = \mu = 0 $) of a flat Riemannian manifold.
The following formulas are also valid for a non-Sasakian $ (k, \mu)$-contact manifolds (e.g., \cite{blair1995contact}):
\begin{eqnarray}
\label{2.11}
 && h^2X = (k-1)\varphi^2X , \\
\label{2.12}
 && QX = [2(n - 1) - n\mu ]X + [2(n - 1) + \mu ] hX + [2(1 - n) + n(2k + \mu)]\eta(X)\,\xi
\end{eqnarray}
for any $X\in\mathfrak{X}_M$. Equation \eqref{2.12} gives the constant scalar curvature
\begin{equation}\label{2.12b}
 r = 2n(2(n - 1) + k - n\mu)
\end{equation}
and $Q\xi=2\,nk\,\xi$.
Further, (\ref{2.11}) shows us that $k \le 1$, and the equality $k = 1$ holds when $M$ is Sasakian. For the non-Sasakian case, i.e., $k < 1$, the ($k,\mu$)-nullity condition determines the curvature of $M$ completely. By virtue of this, Boeckx \cite{boeckx2000full} proved that a non-Sasakian ($k,\mu$)-contact manifold is locally homogeneous and, therefore, analytic.
The {divergence of a tensor field} $K$ of type $(1,s)$ is defined by
\[
 (\Div K)(X_1,X_2,\dots,X_s)=\sum\nolimits_{\,i=1}^{2n+1} g((\nabla_{e_i}K)(X_1,X_2,\dots,X_s),e_i)
\]
for all $X_1,X_2,\dots,X_s \in\mathfrak{X}_M$.
It is well known that ${\cal L}_{X}f = X(f) = g(X,\nabla f)$ for an arbitrary $X\in\mathfrak{X}_M$ and any smooth function $f$,
where $\nabla f$ is the gradient of $f$. The {Hessian} of a smooth function $f$ on $M$ is defined for all $X,Y\in\mathfrak{X}_M$ by
\[
 {\rm Hess}_f(X,Y) = \nabla^2f(X,Y) = g(\nabla_{X}\nabla f,Y).
\]

\section{On closed $(m,\rho)$-quasi-Einstein manifolds}

Here, we consider an
$(m,\rho)$-quasi-Einstein structure with a closed $1$-form $V^{\flat}$
on a contact metric manifold.
Notice that if a $1$-form $V^{\flat}$ is closed then \eqref{1.2} reduces to
\eqref{1.1} with $V^{\flat}=\nabla f$, where
$f$ is a smooth function on $M$,
therefore, it is an important generalization of a gradient Ricci almost soliton and a gradient $m$-quasi-Einstein manifold.
Before proceeding to the main results, we prove the following.

\begin{lemma}
The Ricci tensor $\Ric(V,X)$ for any $X\in\mathfrak{X}_M$
on $(2n+1)$-dimensional closed $(m,\rho)$-quasi-Einstein manifold $M$ is given by
\begin{eqnarray}\label{C}
 \frac{m-1}{m}\,\Ric(V,X) = \frac{2n\lambda+r(2n\rho -1)}{m}\,V^{\flat}(X) -\frac{4n\rho-1}{2}\,X(r).
\end{eqnarray}
\end{lemma}

\begin{proof} By condition, $V^{\flat}$ is closed, thus, \eqref{1.2} can be rewritten as
\begin{eqnarray}\label{A}
\nabla_{X}V = \frac{1}{m}V^{\flat}(X)V -QX + (\lambda+\rho\,r)  X
\end{eqnarray}
for any $X\in\mathfrak{X}_M$. Taking the covariant derivative of \eqref{A} along any $Y\in\mathfrak{X}_M$, we get
\begin{eqnarray*}
 && \nabla_{X}\nabla_{Y}V-\nabla_{Y}\nabla_{X}V-\nabla_{[X,Y]}V = (\nabla_{Y}Q)X - (\nabla_{X}Q)Y  \\
 && -\,\frac{1}{m}\{V^{\flat}(Y)\nabla_{X}V-V^{\flat}(X)\nabla_{Y}V\} -\rho\{Y(r)X-X(r)Y\}
\end{eqnarray*}
for any $X$, $Y\in\mathfrak{X}_M$. Applying \eqref{A} and the expression
$R(X,Y) = [\nabla_{X},\nabla_{Y}]-\nabla_{[X,\,Y]}$ of the curvature tensor, we obtain
\begin{eqnarray}\label{B}
\nonumber
 && R(X,Y)V = (\nabla_{Y}Q)X - (\nabla_{X}Q)Y - \frac{1}{m}\{V^{\flat}(Y)QX-V^{\flat}(X)QY\} \\
 && +\,\frac{1}{m}(\lambda+\rho\,r)\{V^{\flat}(Y)X - V^{\flat}(X)Y\}  -\rho\{Y(r)X-X(r)Y\}
\end{eqnarray}
for any $X$, $Y\in\mathfrak{X}_M$. Contracting \eqref{B} over $Y$ with respect to a local orthonormal frame $\{e_i\}_{1\leq i\leq 2n+1}$,
we find
\begin{eqnarray*}
 \frac{m-1}{m}\,\Ric (V,X) &=& \sum\nolimits_{i=1}^{2n+1}g((\nabla_{X}Q)e_i-(\nabla_{e_{i}}Q)X,e_{i})\\
 &+&\frac{1}{m}\,(2n\lambda+r(2n\rho -1))V^{\flat}(X)-2n\rho\,X(r)
\end{eqnarray*}
for any $X\in\mathfrak{X}_M$. Contracting the second Bianchi's identity, yields the known formula $\Div Q=\frac{1}{2}\nabla r$,
see \cite{besse2007einstein}.
Applying this in the preceding equation, we obtain the required result.
\end{proof}

\begin{theorem}\label{thm3.1}
Let a
$K$-contact manifold $M^{2n+1}(\varphi,\xi,\eta,g)$ admit a closed $(m,\rho)$-quasi-Ein\-stein structure with $m\ne1$. Then $M$ is a
Einstein
manifold of constant scalar curvature $2n(2n+1)$.
\end{theorem}

\begin{proof} Substituting $\xi$ for $Y$ in \eqref{B} and then taking inner product with $Y\in \mathfrak{X}_M$, we have
\begin{eqnarray}\label{3.1}
\nonumber
 && g(R(X, \xi,)V,Y) = g((\nabla_{\xi}Q)X,Y)-g((\nabla_{X}Q)\,\xi,Y) \\
\nonumber
 && -\,\frac{1}{m}\,\eta(V)\{g(QX,Y)-(\lambda+\rho\,r)g(X,Y)\}\\
 && -\{\frac{1}{m}(\lambda+\rho\,r-2n)V^{\flat}(X)-\rho\,X(r)\}\eta(Y)-\rho\,\xi(r)g(X,Y)
\end{eqnarray}
for all $X\in\mathfrak{X}_M$, where we have used the formula for a $K$-contact manifold (e.g., \cite{blair2010riemannian}),
\begin{equation}\label{Eq-Q}
 Q\,\xi = 2n\,\xi .
\end{equation}
Taking covariant derivative \eqref{Eq-Q} along $X\in \mathfrak{X}_M$ and using \eqref{2.6}, we get
\begin{eqnarray}\label{3.2}
 (\nabla_{X}Q)\,\xi = Q\varphi\,X -2\,n\varphi\,X.
\end{eqnarray}
Further, since $\xi$ is a Killing field, we have ${\cal L}_{\xi}\Ric=0$; thus, it follows from \eqref{3.2} and \eqref{2.6} that
\[
 \nabla_{\xi}\,Q=Q\varphi-\varphi\,Q.
\]
This implies $\xi(r)=0$. Thus, inserting $X=\varphi\,X$ and $Y=\varphi\,Y$ in \eqref{3.1} and then adding the resulting equation with \eqref{3.1},
we obtain
\begin{eqnarray}\label{3.3}
\nonumber
 && g(R(\varphi\,X, \xi,)V,\varphi\,Y)+g(R(X, \xi,)V,Y)=4ng(\varphi\,X,Y)\\
\nonumber
 && -g(Q\varphi\,X,Y)-g(\varphi QX,Y)  -\frac{1}{m}\,\eta(V)\,\{g(QX,Y)+g(Q\varphi\,X,\varphi\,Y)\}\\
\nonumber
 && +\,\frac{\lambda+\rho\,r}{m}\eta(V) \{g(X,Y)-\eta(X)\eta(Y)\}\\
 && -\,\{\frac{1}{m}\,(\lambda+\rho\,r-2n)V^{\flat}(X)-\rho\,X(r)\}\eta(Y)
\end{eqnarray}
for all $X$, $Y\in \mathfrak{X}_M$, where we have used \eqref{2.1}, \eqref{Eq-Q} and \eqref{3.2}.
Next, recall two formulas concerning $K$-contact manifolds (see \cite[p. 95]{blair2010riemannian}):
\begin{eqnarray*}
 R(\xi,X)Y \eq (\nabla_{X}\,\varphi)Y,\\
 (\nabla_{\varphi\,X}\varphi)\varphi\,Y+(\nabla_{X}\varphi)Y \eq 2g(X,Y)\,\xi-\eta(Y)(X + \eta(X)\,\xi)
\end{eqnarray*}
for all $X,\,Y\in \mathfrak{X}_M$. Applying these formulas in \eqref{3.3}, we acquire
\begin{eqnarray*}
 && \frac{1}{m}\,\eta(V)\{g(QX,Y)+g(Q\varphi\,X,\varphi\,Y)\}+g((Q\varphi+\varphi Q)X,Y)+2\eta(V) g(X,Y) \\
 && -\,\eta(V)\eta(X)\eta(Y) = 4ng(\varphi\,X,Y)+\frac{1}{m}\,\eta(V)(\lambda+\rho\,r)\{2g(X,Y)-\eta(X)\eta(Y)\} \\
 && -\,\{\frac{1}{m}(\lambda+\rho\,r-2n-m)V^{\flat}(X)-\rho\,X(r)\}\eta(Y)
\end{eqnarray*}
for all $X$, $Y\in \mathfrak{X}_M$, where we have used \eqref{Eq-Q}. Anti-symmetrizing the preceding equation provides
\begin{eqnarray}\label{3.4}
 && \frac{1}{m}(\lambda+\rho\,r -2n-m)\{V^{\flat}(X)\eta(Y)-V^{\flat}(Y)\eta(X)\} \nonumber\\
 && = 8ng(\varphi\,X,Y) -2g((Q\varphi+\varphi Q)X,Y) +\rho\{X(r)\eta(Y)-Y(r)\eta(X)\}
\end{eqnarray}
for all $X$, $Y\in \mathfrak{X}_M$. Next, replacing $X$ by $\varphi\,X$ and $Y$ by $\varphi\,Y$ and using \eqref{2.1}, yields
\begin{eqnarray}\label{3.5}
 Q\varphi\,X+\varphi QX= 4\,n\varphi\,X
\end{eqnarray}
for any $X\in \mathfrak{X}_M$. Consider an orthonormal $\varphi$-basis $\{e_{i},\varphi e_{i},\xi;\ i=1,\ldots,n\}$ on $M$ satisfying
$Qe_{i} = \nu_{i}e_{i}$. It follows from \eqref{3.5} that $Q\varphi e_{i}  = (4n-\nu_{i})\varphi e_{i}$.
Therefore, the scalar curvature $r$ is given by
\[
 r = \Ric(\xi,\xi) + \sum\nolimits_{i=1}^{n}\{g(Qe_{i},e_{i}) + g(Q\varphi e_{i},\varphi e_{i})\}.
\]
Applying \eqref{Eq-Q}, we acquire $r =2n(2n+1)$. In view of this and \eqref{3.5}, equation \eqref{3.4} reduces to
\begin{eqnarray}\label{3.6}
 \frac{1}{m}(\lambda+\rho\,r -2n-m)\{V^{\flat}(X)\,\eta(Y)-V^{\flat}(Y)\,\eta(X)\} = 0
\end{eqnarray}
for all $X$, $Y\in \mathfrak{X}_M$. Since $\lambda,\rho,r$ and $m$ are all constants, inserting $Y=\xi$ in \eqref{3.6} shows that
two cases are possible: $(i)\ V=\eta(V)\,\xi$, or $(ii)\ \lambda+\rho\,r -2n-m=0$.

\smallskip

\textbf{Case I}. Taking covariant derivative of $V=\eta(V)\,\xi$ along $X\in \mathfrak{X}_M$ and using \eqref{2.6} yields
\[
 \nabla_{X}V = \eta(\nabla_{X}V)\,\xi -\eta(X)\,\varphi\,X.
\]
In view of this and since $V^{\flat}$ is closed, we find
\[
 \eta(\nabla_{X}V)\,\eta(Y)-\eta(\nabla_{Y}V)\,\eta(X)+2\eta(V)\,d\eta(X,Y)=0
\]
for all $X$, $Y\in \mathfrak{X}_M$. Inserting $X=\varphi\,X$ and $Y=\varphi\,Y$ and using \eqref{2.1} leads to
\[
 \eta(V)\,d\eta(X,Y)=0
\]
for all $X$, $Y\in \mathfrak{X}_M$. This implies $\eta(V)=0$, as $d\eta$ is non-vanishing on $M$. Thus, $V=\eta(V)\,\xi=0$, therefore, $(M,g)$ is an Einstein manifold with positive scalar curvature $2n(2n+1)$.

\smallskip

\textbf{Case II}. In this case, \eqref{C} reduces to $(m-1)(QV-2nV)=0$, where we have used the fact that $r=2n(2n+1)$ is a constant.
This directly implies $QV=2nV$, as $m\ne 1$. Taking its covariant derivative along $X\in\mathfrak{X}_M$ and using \eqref{A}, we obtain
\begin{eqnarray*}
 (\nabla_XQ)V-Q^2X+\frac{1}{m}V^{\flat}(X)QV+(\lambda+\rho\,r+2n)QX \\
 =\frac{2n}{m}\{V^{\flat}(X)V+(\lambda+\rho\,r)m X\}.
\end{eqnarray*}
Substituting $QV=2nV$ into the foregoing equation, we acquire
\begin{eqnarray}\label{3.10}
(\nabla_XQ)V-Q^2X+(\lambda+\rho\,r+2n)QX=2n(\lambda+\rho\,r) X
\end{eqnarray}
for any $X\in\mathfrak{X}_M$. Contracting \eqref{3.10} over $X$
with respect to a local orthonormal frame $\{e_i\}_{1\le i\le 2n+1}$ and using $r=2n(2n+1)$, we obtain
\begin{equation}\label{3.11}
(\Div Q)V- ||Q||^2 +2nr=0.
\end{equation}
Since the scalar curvature $r=2n(2n+1)$ is constant, it follows from the formula $\Div Q=\frac{1}{2}\,\nabla r$
(obtained by the contraction of Bianchi's second identity, see \cite{besse2007einstein}), that $\Div Q=0$. Thus, equation \eqref{3.11} gives
 $||Q||^2= 4n^2(2n+1)$.
Using this and $r=2n(2n+1)$, we find from \eqref{3.11}
\begin{eqnarray*}
 \big\| Q-\frac{r}{2n+1}\,I\big\|^2 = ||Q||^2-\frac{2r^2}{2n+1}+\frac{r^2}{2n+1}=0,
\end{eqnarray*}
that is,
the symmetric tensor $Q-\frac{r}{2n+1}\,I$
vanishes. This implies $\Ric=2n\,g$. Hence, $M$ is an Einstein manifold with Einstein constant $2n$.
\end{proof}

\begin{corollary}\label{cor3.1}
If, in conditions of Theorem~\ref{thm3.1}, $(M,g)$ is complete, then it is a compact Einstein Sasakian manifold
 of constant scalar curvature $2n(2n+1)$.
\end{corollary}

\begin{proof}
In case I, the compactness of $M$ follows from Myers' Theorem \cite{myers1935connections}, since $M$ is complete.
Using the result of Boyer-Galicki \cite{boyer2001einstein}: \textit{"Any compact $K$-contact Einstein manifold is Sasakian"},
we conclude that $M$ is a Sasakian manifold.
In case II,
\eqref{A} reduces to
 $\nabla_{X}V = mX+\frac{1}{m}V^{\flat}(X)V$
for any $X\in\mathfrak{X}_M$. Tracing this, we get
\begin{equation}\label{3.12}
 m\,\Div\,V=||V||^2+m^2(2n+1).
\end{equation}
Since $(M,g)$ is a complete Einstein manifold of positive Ricci curvature, by Myers' Theorem \cite{myers1935connections}, it is compact.
Integrating \eqref{3.12} over $M$ and applying the divergence Theorem, yields
\[
 \int_{M}||V||^2d\vol =-m^2(2n+1)\,{\rm Vol}(M,g).
\]
We conclude that $V=0$, therefore, from \eqref{3.12} we acquire $m^2(2n+1)=0$. This implies $m=0$ -- a contradiction, as $m$ is positive.
\end{proof}

It is known that the Ricci operator $Q$ and $\varphi$ commute on a Sasakian manifold, i.e., $Q\varphi=\varphi Q$, see \cite{blair2010riemannian}.
Thus, the following result follows from \eqref{3.5} and Corollary~\ref{cor3.1}.

\begin{corollary}\label{cor3.2}
Let a (complete) Sasakian manifold $M^{2n+1}(\varphi,\xi,\eta,g)$ admit a closed $(m,\rho)$-quasi-Ein\-stein structure. Then $M$ is a
(compact) Einstein manifold of constant scalar curvature $2n(2n+1)$.
\end{corollary}

\begin{remark}\rm
Ghosh recently considered the $K$-contact metrics as a (gradient) $m$-quasi-Einstein structure \cite{ghosh2015quasi}, a (gradient) $(m,\rho)$-quasi-Einstein structure \cite{ghosh2014m} and an $m$-quasi-Einstein structure \cite{ghosh2019m}. If we choose $V=\nabla f$, then a quasi-Einstein structure reduces to a (gradient) quasi-Einstein structure. Thus, an $(m,\rho)$-quasi-Einstein structure, where the scalar curvature $r$ and a constant $\rho$ are extra imposed in \eqref{1.2}, is the generalization of a (gradient) $m$-quasi-Einstein, a (gradient) $(m,\rho)$-quasi-Einstein and an $m$-quasi-Einstein structures.
Consequently, the results proved by Ghosh in \cite{ghosh2019m,ghosh2014m,ghosh2015quasi} for $K$-contact and Sasakian manifolds follow from our Theorem \ref{thm3.1}, Corollary~\ref{cor3.1} and Corollary \ref{cor3.2}.
\end{remark}


\begin{theorem}\label{thm3.3}
Let a non-Sasakian $(k,\mu)$-contact manifold $M^{2n+1}(\varphi,\xi,\eta,g)$ admit a closed $(m,\rho)$-quasi-Einstein structure.
Then $M$ is flat when $n=1$, and for $n>1$, $M$ is locally isometric to the product of a Euclidean space $\RR^{n+1}$ and a sphere $S^n$ of constant curvature $4$.
\end{theorem}

\begin{proof}
Taking covariant derivative of $Q\xi=2nk\xi$ (follows from \eqref{2.12}) along $X\in\mathfrak{X}_M$ and using (\ref{2.1}), we acquire
\begin{eqnarray}\label{3.13}
 (\nabla_{X}Q)\,\xi = Q(\varphi +\varphi\,h)X - 2nk(\varphi + \varphi\,h)X.
\end{eqnarray}
Taking the inner product of \eqref{B} with $\xi$ and using \eqref{3.13}, we obtain
\begin{eqnarray}\label{3.14}
 && g(R(X,Y)V, \xi) = g(Q(\varphi+\varphi\,h)Y,X)-g(Q(\varphi+\varphi\,h)X,Y)\nonumber\\
 && -\,4nkg(\varphi\,X,Y) +\frac{1}{m}(\lambda+\rho\,r-2nk)\{V^{\flat}(Y)\eta(X)-V^{\flat}(X)\eta(Y)\}
\end{eqnarray}
for all $X$, $Y\in\mathfrak{X}_M$, where we have used the fact that the scalar curvature is constant on a $(k,\mu)$-contact manifold
and $h\,\varphi=-\varphi\,h$ holds. Substituting $Y=\xi$ in \eqref{3.14} and using \eqref{2.1} and \eqref{2.10}, we~get
\begin{equation*}
\mu V^{\flat}(hX)=\frac{1}{m}(\lambda+\rho\,r-2nk-k)\{V^{\flat}(X)-\eta(V)\eta(X)\}
\end{equation*}
for any $X\in\mathfrak{X}_M$, where we have used $Q\xi=2nk\xi$ and that $h$ is a symmetric operator.
This implies
\begin{eqnarray}\label{ZZ}
 m\mu hV=\nu(V-\eta(V)\,\xi),
\end{eqnarray}
where $\nu=\lambda+\rho\,r-2nk-k$. Differentiating this along $\xi$ and applying \eqref{A}, we get
\begin{eqnarray}\label{3.16}
 m\mu\{(\nabla_{\xi}h)V+\frac{1}{m}\eta(V)hV\}=\frac{\nu\eta(V)}{m}\{V-\eta(V)\,\xi\},
\end{eqnarray}
where we have used that $Q\xi=2nk\xi$ and $\nabla_{\xi}\,\xi=0$ -- follows from \eqref{2.3}.
Recall a formula concerning contact metric manifolds (e.g., \cite{blair2010riemannian}):
\[
 (\nabla_{\xi}h)X = \varphi\,X - \varphi\,h^2 X - \varphi\,l X
\]
for any $X\in\mathfrak{X}_M$. In view of \eqref{2.1}, \eqref{2.10} and \eqref{2.11}, the foregoing equation becomes
\begin{eqnarray*}
(\nabla_{\xi}h)X = \varphi\,X - \{(k-1)\varphi^2\}\varphi\,X - \varphi \{-k \varphi^2 X + \mu h X\}
\end{eqnarray*}
for any $X\in\mathfrak{X}_M$. Applying \eqref{2.1} and $\varphi\,h=-h\,\varphi$ in the last equation, we have
\[
 (\nabla_{\xi}h)X =\mu h\,\varphi\,X
\]
for any $X\in\mathfrak{X}_M$. By this and equality \eqref{ZZ},
equation \eqref{3.16} becomes $\mu^2 h\,\varphi V=0$.
This implies two possible cases: $(i)$ $\mu=0$, or $(ii)$ $\mu \ne 0$.

\smallskip

\textbf{Case I}. First, replacing $X$ by $\varphi\,X$ and $Y$ by $\varphi\,Y$ in \eqref{3.14} and using \eqref{2.1} \eqref{2.10}, we obtain
\begin{eqnarray}\label{3.17}
hQ\varphi\,X + \varphi QhX - Q\varphi\,X - \varphi QX + 4nk \varphi\,X =0
\end{eqnarray}
for any $X\in\mathfrak{X}_M$, where we have used $Q\xi=2nk\xi$, $h\xi=0$ and $h\,\varphi=-\varphi\,h$. Applying \eqref{2.12} repeatedly in \eqref{3.17} and using \eqref{2.1}, we acquire $k(\mu-2)=\mu(n+1)$, therefore, $\mu=k=0$. It follows from \eqref{2.12} that $R(X,Y)\,\xi = 0$ for any $X$, $Y\in\mathfrak{X}_M$. Hence, $M$ is flat when $n=1$, and, for higher dimensions, $M$ is locally isometric to the product $\RR^{n+1}\times S^n(4)$, see \cite[p.~122]{blair2010riemannian}.

\smallskip

\textbf{Case II}. In this case, we have $h\,\varphi V=0$. Operating this by $h$ and using \eqref{2.1} and \eqref{2.11} yields $(k-1)\varphi V=0$.
Since $M$ is non-Sasakian, we have $\varphi V=0$. Again, operating by $\varphi$ and using \eqref{2.1}, gives $V=\eta(V)\,\xi$. Taking its covariant derivative along $X\in \mathfrak{X}_M$ and using \eqref{2.3}, provides
\[
 \nabla_{X}V = \eta(\nabla_{X}V)\,\xi -\eta(X)(\varphi\,X+\varphi\,hX).
\]
Using this, we obtain
\begin{equation*}
 dV^{\flat}(X,Y)=\eta(\nabla_{X}V)\eta(Y)-\eta(\nabla_{Y}V)\eta(X)+2\eta(V)g(X,\varphi\,Y)
\end{equation*}
for all $X$, $Y\in \mathfrak{X}_M$. As $V^{\flat}$ is closed, choosing $X,Y \bot\, \xi$, we get $\eta(V)\,d\eta(X,Y)=0$ for all $X, Y\in\mathfrak{X}_M$.
This implies $\eta(V)=0$, as $d\eta$ is non-vanishing on a contact manifold, therefore, $V=0$. It follows from \eqref{A} that
\begin{equation}\label{3.18b}
 QX=(\lambda+\rho\,r)X=2nk X.
\end{equation}
This gives the scalar curvature $r=2nk(2n+1)$. Comparing this with the scalar curvature in \eqref{2.12b}, we get $n\mu=2(n-1)-2nk$.
Thus, using \eqref{3.18b} and (\ref{2.12}), we get
\[
 (2(n-1)+\mu)hX=0
\]
for any $X\in\ker\eta$. Since $M$ is a non-Sasakian manifold, we have $2(n-1)+\mu=0$. Thus, for $n=1$ we have $\mu=0$, therefore, from relation $k(\mu-2)=\mu(n+1)$ (obtained in Case I), we achieve $k=0$.
It follows from \eqref{2.12} that $R(X,Y)\,\xi=0$. Hence, $M$ is locally flat when $n=1$. Next, for $n>1$, combining $2(n-1)+\mu =0$ and $k(\mu-2)=\mu(n+1)$, we find  $k=n-\frac{1}{n}>1$ -- a contradiction.
\end{proof}

\begin{remark}\rm
Here, our characterization mainly depends on the potential vector field $V$ and does not depend on a smooth function $f$. Therefore, if we consider $V=\nabla f$, then a (gradient) $m$-quasi-Einstein (i.e., satisfying \eqref{1.1} for $\rho=0$) and a (gradient) $(m,\rho)$-quasi-Einstein structures with the potential function $f$ reduce to a particular case, and the result proved by Ghosh in \cite{ghosh2015quasi} for $(k,\mu)$-contact manifolds is compatible with our Theorem~\ref{thm3.3}. In particular, Theorem~\ref{thm3.3} is true for a (gradient) $(m,\rho)$-quasi-Einstein structure.
\end{remark}


\section{On $(m,\rho)$-quasi-Einstein manifolds}

Here, we study non-gradient $(m,\rho)$-quasi-Einstein contact metric manifolds. Suppose that a $K$-contact manifold $M^{2n+1}(\varphi,\xi,\eta,g)$ admits a $(m,\rho)$-quasi-Einstein structure, whose potential vector field $V$ is collinear to the Reeb vector field $\xi$, i.e., $V=\sigma\xi$ for some smooth function $\sigma$ on $M$. Differentiating this along $X\in \mathfrak{X}_M$ and using \eqref{2.6}, yields
\[
 \nabla_X V = (X \sigma)\,\xi -\sigma(\varphi\,X).
\]
Thus, \eqref{1.2} reduces to
\begin{equation}\label{K}
 \frac{1}{2}\{X(\sigma)\,\eta(Y)+(Y\sigma)\,\eta(X)\} + \Ric(X,Y) -\frac{\sigma^2}{m}\,\eta(X)\,\eta(Y)=(\lambda+\rho\,r)\, g(X,Y)
\end{equation}
for all $X$, $Y\in\mathfrak{X}_M$. Replacing $Y$ by $\xi$ in \eqref{K} and using \eqref{Eq-Q},
we get
\begin{equation}\label{KK}
 X(\sigma)+\xi(\sigma)\,\eta(X)=2\,(\lambda+\rho\,r+\frac{\sigma^2}{m}-2n)\,\eta(X)
\end{equation}
for any $X\in\mathfrak{X}_M$. Setting $X=Y=\xi$ in \eqref{K} and using \eqref{Eq-Q},
yield $\xi(\sigma)=\lambda+\rho\,r+\frac{\sigma^2}{m}-2n$.
Substituting this into \eqref{KK}, we acquire $\nabla \sigma=\xi(\sigma)\,\xi$.
Taking its covariant derivative along $X\in\mathfrak{X}_M$ and then inner product with $Y\in\mathfrak{X}_M$, we obtain
\[
 g(\nabla_X\nabla\sigma,Y)=X(\xi(\sigma))\,\eta(Y)-\xi(\sigma)\,g(\varphi\,X,Y).
\]
Since ${\rm Hess}_{\sigma}$ is symmetric, i.e., $g(\nabla_X\nabla\sigma,Y)=g(\nabla_Y\nabla\sigma,X)$, we conclude that
$\xi(\sigma)\,d\eta(X,Y)=0$ for all $X,Y$ orthogonal to $\xi$. This shows us that $\xi(\sigma)=0$ on $M$, as $d\eta$ is non-vanishing  on $M$.

Thus, $\nabla\sigma=0$ and therefore $\sigma$ is constant on $M$. It follows from \eqref{K} that $M$ is an $\eta$-Einstein manifold.
To extend this observation for a contact metric manifold, applying the
modified divergence Theorem (see Lemma~\ref{L-Div-1} in what follows), we prove the following.

\begin{theorem}\label{thm3.4}
Let a complete
contact metric manifold $M(\varphi,\xi,\eta,g)$ admit an $(m,\rho)$-quasi-Einstein structure, whose potential vector field $V$ is collinear to $\xi$,
i.e., $V=\sigma\,\xi$ for a smooth function $\sigma$ on $M$,
and $\|\nabla(\sigma^2)-\frac{4}{3m}\,\sigma^2 V +2(2n-1)\,\sigma\rho\,r\,\xi\,\|_g\in {\rm L}^1(M,g)$.
Then $M$ is a $K$-contact $\eta$-Einstein manifold; moreover, $\sigma$ is a constant.
\end{theorem}

\begin{proof}
Taking covariant derivative of the equality $V=\sigma\,\xi$ along $X\in \mathfrak{X}_M$ and using \eqref{2.3}, we acquire
\begin{eqnarray}\label{3.19A}
\nabla_{X}V = X(\sigma)\,\xi -\eta(X)(\varphi\,X+\varphi\,hX).
\end{eqnarray}
Contracting this over $X$ and using $\tr_g\varphi=\tr_g(\varphi\,h)=0$, we have $\Div V=\xi(\sigma)$. Using this in the contraction of $\nabla_X(\sigma^2V)=X(\sigma^2)V+\sigma^2(\nabla_XV)$ over $X$, yields
\begin{eqnarray}\label{3.19B}
 \Div\,(\sigma^2V)=g(\nabla\sigma^2,V)+\sigma^2\Div V=3\,\sigma^2\,\xi(\sigma).
\end{eqnarray}
Now, by virtue of \eqref{3.19A}, equation \eqref{1.2} can be written as
\begin{eqnarray}\label{3.20}
 \frac{1}{2}\,\{Y(\sigma)\,\xi+\eta(Y)\nabla\sigma\}+QY -\sigma\varphi\,hY - \frac{\sigma^2}{m} \eta(Y)\,\xi= (\lambda+\rho\,r) Y
\end{eqnarray}
for any $Y\in \mathfrak{X}_M$. Taking covariant derivative of \eqref{3.20} along $X\in \mathfrak{X}_M$ and using \eqref{2.3}, we obtain
\begin{eqnarray*}
 \frac{1}{2}\,\{g(Y,\nabla_{X}\nabla\sigma)\,\xi+Y(\sigma)(\varphi\,X+\varphi\,hX)-g(Y,\varphi\,X+\varphi\,hX)\nabla\sigma
 +\eta(Y)\nabla_{X}\nabla\sigma\}\nonumber\\
 +\,(\nabla_{X}Q)Y -X(\sigma)\varphi\,hY - \sigma(\nabla_{X}\varphi\,h)Y= \rho\,X(r)Y+ \frac{2\sigma}{m}X(\sigma)\eta(X)\,\xi\\
 -\,\frac{\sigma^2}{m}\{\eta(Y)(\varphi\,X+\varphi\,hX)+g(Y,\varphi\,X+\varphi\,hX)\,\xi\}
\end{eqnarray*}
for any $Y\in \mathfrak{X}_M$. Contracting this over $X$
with respect to a local orthonormal frame $\{e_i: i=1,\ldots,2n+1\}$ and using formulas $\Div Q=\frac{1}{2}\nabla r$, $\tr_g\varphi=\tr_g\varphi\,h=0$, provides
\begin{eqnarray}\label{3.22}
\nonumber
 && \frac{1}{2}\,\{g(Y,\nabla_{\xi}\nabla\sigma)+g(\varphi\,Y-\varphi\,hY,\nabla\sigma)+\eta(Y)\Div\,(\nabla\sigma)\} \\
 && -\,\frac{2\rho-1}{2}\,g(Y,\nabla r) = g(\varphi\,hY,\nabla\sigma) + \sigma \Div\,(\varphi\,h)+ \frac{2\sigma}{m}\,\xi(\sigma)\eta(Y)
\end{eqnarray}
for any $Y\in \mathfrak{X}_M$. Recall the following formula for a contact metric manifold (see \cite{blair2010riemannian}):
\[
 \Div\,(h\,\varphi)X=g(Q\xi,X)-2n\eta(X).
\]
Taking covariant derivative of $g(\xi,\nabla \sigma)=\xi(\sigma)$ along $\xi$ and using $\nabla_{\xi}\,\xi=0$ (follows from \eqref{2.3}), we get $g(\xi,\nabla_{\xi}\nabla\sigma)=\xi(\xi(\sigma))$. Thus, replacing $\xi$ for $Y$ in \eqref{3.22}, we achieve
\begin{equation}\label{3.23}
\frac{1}{2}\,\{\Div\,(\nabla\sigma)+\xi(\xi(\sigma)) - (2\rho-1)\,\xi(r)\}= \frac{2\sigma}{m}\,\xi(\sigma)-\sigma\{\Ric(\xi,\xi)-2n\} .
\end{equation}
Further, contracting \eqref{3.20} and using the equality $\tr_g\varphi\,h=0$, we acquire
\begin{equation*}
 r+\xi(\sigma)-\frac{1}{m}\,\sigma^2 =(2n+1)(\lambda+\rho\,r).
\end{equation*}
It follows that
$\xi(\xi(\sigma))=\frac{1}{m}\,\xi(\sigma^2)+((2n+1)\rho-1)\,\xi(r)$.
Substituting this into \eqref{3.23} and using \eqref{2.4b}, we get
\begin{eqnarray}\label{3.25}
 \frac{1}{2}\,\{\Div\,(\nabla\sigma)+(2n-1)\rho\,\xi(r)\}= \frac{\sigma}{m}\,\xi(\sigma)+\sigma\,||h||^2.
\end{eqnarray}
Applying $(\Delta\,\sigma)\,\sigma=\frac{1}{2}\,\Delta(\sigma^2) +||\nabla\sigma||^2$, see \cite{yano1970integral},
$\Div(r\,\xi)=r\,\Div\xi+\xi(r)$ (with $\Div\xi=-H=0$, see Lemma~\ref{L-01})
and \eqref{3.19B} in \eqref{3.25}, we derive
\begin{equation*}
 \Div\big(\,\nabla(\sigma^2)-\frac{4}{3m}\,\sigma^2 V +2(2n-1)\,\sigma\rho\,r\,\xi\big)
 =4\,\sigma^2 ||h||^2 +2\,||\nabla\sigma||^2,
\end{equation*}
where we have used the convention $\Div\,(\nabla\sigma)=-\Delta\,\sigma$.
Applying Lemma~\ref{L-Div-1}, we infer that
\[
 2\,\sigma^2 ||h||^2+||\nabla\sigma||^2 = 0.
\]
This implies $h=0$ and $\nabla\sigma=0$;
hence, $M$ is a $K$-contact manifold and $\sigma$ is constant. Thus, we conclude from \eqref{3.20} that $M$ is an $\eta$-Einstein manifold.
\end{proof}

A direct consequence of Theorem~\ref{thm3.4} in the case of compactness of $M$ is the following statement,
which for a particular case of $\rho=0$ (i.e., $m$-quasi-Einstein manifolds) was proved in \cite[Theorem~5.2]{ghosh2019m}.

\begin{corollary}
If a compact contact metric manifold $M(\varphi,\xi,\eta,g)$ admits an $(m,\rho)$-quasi-Einstein structure, whose potential vector field $V$ is collinear to $\xi$,
then $M$ is a $K$-contact $\eta$-Einstein mani\-fold; moreover, $\sigma$ is a constant.
\end{corollary}


Modifying divergence Theorem on a complete Riemannian manifold $(M,g)$ yields the following.

\begin{lemma}[see Proposition~1 in \cite{csc2010}]\label{L-Div-1}
Let $(M^n,g)$ be a complete open Riemannian manifold endowed with a vector field $X$
such that $\Div X\ge0$. If the norm $\|X\|_g\in{\rm L}^1(M,g)$, then $\Div X\equiv0$.
\end{lemma}

A contact metric manifold is said to be an $H$-\textit{contact manifold} if the Reeb vector field $\xi$ is harmonic. In \cite{perrone2004contact}, Perrone proved that \textit{``A contact metric manifold is an $H$-contact manifold,
if and only if $\xi$ is an eigenvector of the Ricci operator."}
On a contact metric manifold, the assumption that $\xi$ is an eigenvector of the Ricci operator implies
\begin{equation}\label{Eq-Q1}
  Q\,\xi=(\tr_gl)\,\xi.
\end{equation}
In this context, we ask the following question: \textit{``When a non-gradient
$(m,\rho)$-quasi-Einstein $H$-contact manifold is an Einstein manifold?"}.
If we consider an $H$-contact manifold admitting an $(m,\rho)$-quasi-Einstein structure and the non-zero potential vector field $V=\sigma\xi$ along $X\in \mathfrak{X}_M$, then, by Theorem~\ref{thm3.4}, \eqref{3.19A} and \eqref{3.20} are also valid. Thus, using
\eqref{Eq-Q1} in \eqref{3.20}, we acquire the following:
\begin{equation*}
 \tr_gl+\xi(\sigma)-\frac{\sigma^2}{m}=\lambda+\rho\,r.
\end{equation*}
Applying this in \eqref{3.20}, we obtain $\nabla\sigma=\xi(\sigma)\,\xi$.
Taking covariant derivative of this along $X\in \mathfrak{X}_M$ and using \eqref{2.3}, gives
\[
 \nabla_{X}\nabla\sigma = X(\xi(\sigma))\,\xi -\eta(X)(\varphi\,X+\varphi\,hX).
\]
From the above, it follows that
\begin{eqnarray}\label{3.28}
 X(\xi(\sigma))\eta(Y)-Y(\xi(\sigma))\eta(X)+2\,\xi(\sigma)g(X,\varphi\,Y)=0
\end{eqnarray}
for all $X$, $Y\in \mathfrak{X}_M$. Replacing $X$ and $Y$ by $\varphi\,X$ and $\varphi\,Y$, respectively, in \eqref{3.28}, we get $\xi(\sigma)=0$,
as $d\eta$ is non-vanishing on $M$, therefore, $\nabla\sigma=0$. Hence, $\sigma$ is constant. Thus, \eqref{3.20} becomes
\begin{eqnarray}\label{3.29}
 QY-\sigma\varphi\,hY - \frac{\sigma^2}{m} \eta(Y)\,\xi = (\lambda+\rho\,r) Y
\end{eqnarray}
for any $Y\in \mathfrak{X}_M$. Taking covariant derivative of this along $X\in \mathfrak{X}_M$, we obtain
\begin{equation*}
 (\nabla_XQ)Y-\sigma(\nabla_X\varphi\,h)Y + \frac{\sigma^2}{m} \{g(\varphi\,X+\varphi\,h X,Y)\,\xi +\eta(Y)(\varphi\,X+\varphi\,h X)\}= \rho\,X(r) Y
\end{equation*}
for any $Y\in \mathfrak{X}_M$. By virtue of the formula on a contact metric manifold (e.g., \cite{blair2010riemannian}):
\[
 \Div((h\,\varphi)Y) = g(QY,\xi)-2\,n\eta(Y),
\]
it follows from foregoing equation that
\begin{eqnarray}\label{3.31}
 \sigma(\tr_gl-2n)\,\eta(Y)=\frac{2\sigma-1}{2}\,g(Y,\nabla r)
\end{eqnarray}
for any $Y\in \mathfrak{X}_M$. Inserting $Y=\varphi\,Y$, we get
\[
 \frac{2\sigma-1}{2}\,g(\varphi\,Y,\nabla r)=0.
\]
This implies $g(\varphi\,Y,\nabla r)=0$ for $\sigma\ne\frac{1}{2}$. Recalling the previous idea, we conclude that the scalar curvature $r$ is constant. This immediately reduces \eqref{3.31} to
\[
 \sigma(\tr_g l -2n)\,\eta(Y)=0
\]
for any $Y\in \mathfrak{X}_M$. This implies $\tr_gl=2n$, as $\sigma$ is non-zero (follows from the condition that $V=\sigma\xi$ is non-zero).
This shows us that $M$ is a $K$-contact manifold, therefore, $V$ is a Killing field.
Next, using \eqref{3.29}, we conclude that $M$ is an $\eta$-Einstein manifold. So, we stated the following.

\begin{theorem}
Let an $H$-contact metric manifold $M^{2n+1}(\varphi,\xi,\eta,g)$ admit an $(m,\rho)$-quasi-Einstein structure with non-zero potential vector field $V=\sigma\,\xi$ for some smooth function $\sigma\ne\frac{1}{2}$. Then $V$ is a Killing vector field and $M$ is a $K$-contact $\eta$-Einstein manifold;
moreover, $\sigma$ is constant.
\end{theorem}
In particular,
note that the class of $H$-contact manifold is sufficiently large, in fact, $K$-contact (in particular, Sasakian), $(k,\mu)$-contact and $\eta$-Einstein contact manifolds provide examples of $H$-contact spaces. Thus, \eqref{3.20} implies the following.

\begin{corollary}
Let a $K$-contact manifold $M^{2n+1}(\varphi,\xi,\eta,g)$ admit a $(m,\rho)$-quasi-Einstein structure, whose potential vector field $V$ is collinear to $\xi$. Then $V$ is a Killing vector field and $M$ is an $\eta$-Einstein manifold.
\end{corollary}

\noindent
 \textbf{Acknowledgments}. The first author was financially supported from BIT Mesra (funded by TEQIP-III, MHRD, Govt. of India).


\end{document}